\documentclass[12pt]{amsart}
\usepackage{amscd,amssymb,latexsym}
\usepackage{fullpage}
\usepackage{xypic}
\newcommand{\Mdef}[2]{\newcommand{#1}{\relax \ifmmode #2 \else $#2$\fi}}

\newcommand{\sm }{\wedge}

\newcommand{\tensor}{\otimes}

\newcommand{\Hom}{\mathrm{Hom}}
\newcommand{\Tor}{\mathrm{Tor}}
\newcommand{\Ext}{\mathrm{Ext}}
\Mdef{\bhom}{\mathbf{\hat{H}om}}

\Mdef{\Mod}{\mathrm{mod}}

\newcommand{\st}{\; | \;}

\newtheorem{thm}{Theorem}[section]
\newtheorem{lemma}[thm]{Lemma}

\newtheorem{cor}[thm]{Corollary}

\theoremstyle{definition}

\newtheorem{remark}[thm]{Remark}

\newcommand{\qqed}{\qed \\[1ex]}
\renewenvironment{proof}[1][\hspace*{-.8ex}]{\noindent {\bf Proof #1:\;}}{\qqed}

\Mdef{\PH} {\Phi^H}
\Mdef{\PK} {\Phi^K}
\Mdef{\PL} {\Phi^L}
\Mdef{\PT} {\Phi^{\T}}

\Mdef{\ef}{E{\cF}_+}
\Mdef{\etf}{\widetilde{E}{\cF}}
\Mdef{\eg}{E{G}_+}
\Mdef{\etg}{\tilde{E}{G}}

\Mdef{\infl}{\mathrm{inf}}
\Mdef{\defl}{\mathrm{def}}
\Mdef{\res}{\mathrm{res}}
\Mdef{\ind}{\mathrm{ind}}
\Mdef{\coind}{\mathrm{coind}}

\Mdef{\univ}{\mathcal{U}}

\Mdef{\Fp}{\mathbb{F}_p}
\Mdef{\Zpinfty}{\Z /p^{\infty}}
\Mdef{\Zpadic}{\Z_p^{\wedge}}

\newcommand{\bi}{\begin{itemize}}
\newcommand{\be}{\begin{enumerate}}
\newcommand{\bc}{\begin{center}}
\newcommand{\bd}{\begin{description}}
\newcommand{\ei}{\end{itemize}}
\newcommand{\ee}{\end{enumerate}}
\newcommand{\ec}{\end{center}}
\newcommand{\ed}{\end{description}}

\newcommand{\lra}{\longrightarrow}
\newcommand{\lla}{\longleftarrow}

\Mdef{\we}{\mathbf{we}}
\Mdef{\fib}{\mathbf{fib}}
\Mdef{\cof}{\mathbf{cof}}
\Mdef{\BI}{\mathcal{BI}}

\newcommand{\hocolim}{\mathop{  \mathop{\mathrm {holim}}\limits_\rightarrow} \nolimits}

\Mdef{\B}{\mathbb{B}}
\Mdef{\C}{\mathbb{C}}
\Mdef{\D}{\mathbb{D}}
\Mdef{\E}{\mathbb{E}}
\Mdef{\T}{\mathbb{T}}
\Mdef{\F}{\mathbb{F}}
\Mdef{\G}{\mathbb{G}}
\Mdef{\I}{\mathbb{I}}
\Mdef{\N}{\mathbb{N}}
\Mdef{\Q}{\mathbb{Q}}
\Mdef{\R}{\mathbb{R}}
\Mdef{\bbS}{\mathbb{S}}
\Mdef{\Z}{\mathbb{Z}}

\Mdef{\bA}{\mathbb{A}}
\Mdef{\bB}{\mathbb{B}}
\Mdef{\bC}{\mathbb{C}}
\Mdef{\bD}{\mathbb{D}}
\Mdef{\bE}{\mathbb{E}}
\Mdef{\bF}{\mathbb{F}}
\Mdef{\bG}{\mathbb{G}}
\Mdef{\bH}{\mathbb{H}}
\Mdef{\bI}{\mathbb{I}}
\Mdef{\bJ}{\mathbb{J}}
\Mdef{\bK}{\mathbb{K}}
\Mdef{\bL}{\mathbb{L}}
\Mdef{\bM}{\mathbb{M}}
\Mdef{\bN}{\mathbb{N}}
\Mdef{\bO}{\mathbb{O}}
\Mdef{\bP}{\mathbb{P}}
\Mdef{\bQ}{\mathbb{Q}}
\Mdef{\bR}{\mathbb{R}}
\Mdef{\bS}{\mathbb{S}}
\Mdef{\bT}{\mathbb{T}}
\Mdef{\bU}{\mathbb{U}}
\Mdef{\bV}{\mathbb{V}}
\Mdef{\bW}{\mathbb{W}}
\Mdef{\bX}{\mathbb{X}}
\Mdef{\bY}{\mathbb{Y}}
\Mdef{\bZ}{\mathbb{Z}}

\Mdef{\cA}{\mathcal{A}}
\Mdef{\cB}{\mathcal{B}}
\Mdef{\cC}{\mathcal{C}}
\Mdef{\mcD}{\mathcal{D}} 
\Mdef{\cE}{\mathcal{E}}
\Mdef{\cF}{\mathcal{F}}
\Mdef{\cG}{\mathcal{G}}
\Mdef{\mcH}{\mathcal{H}} 
\Mdef{\cI}{\mathcal{I}}
\Mdef{\cJ}{\mathcal{J}}
\Mdef{\cK}{\mathcal{K}}
\Mdef{\mcL}{\mathcal{L}}

\Mdef{\cM}{\mathcal{M}}
\Mdef{\cN}{\mathcal{N}}
\Mdef{\cO}{\mathcal{O}}
\Mdef{\cP}{\mathcal{P}}
\Mdef{\cQ}{\mathcal{Q}}
\Mdef{\mcR}{\mathcal{R}}
\Mdef{\cS}{\mathcal{S}}
\Mdef{\cT}{\mathcal{T}}
\Mdef{\cU}{\mathcal{U}}
\Mdef{\cV}{\mathcal{V}}
\Mdef{\cW}{\mathcal{W}}
\Mdef{\cX}{\mathcal{X}}
\Mdef{\cY}{\mathcal{Y}}
\Mdef{\cZ}{\mathcal{Z}}

\Mdef{\ca}{\mathcal{a}}
\Mdef{\ct}{\mathcal{t}}

\Mdef{\At}{\tilde{A}}
\Mdef{\Bt}{\tilde{B}}
\Mdef{\Ct}{\tilde{C}}
\Mdef{\Et}{\tilde{E}}
\Mdef{\Ht}{\tilde{H}}
\Mdef{\Kt}{\tilde{K}}
\Mdef{\Lt}{\tilde{L}}
\Mdef{\Mt}{\tilde{M}}
\Mdef{\Nt}{\tilde{N}}
\Mdef{\Pt}{\tilde{P}}

\Mdef{\tA}{\tilde{A}}
\Mdef{\tB}{\tilde{B}}
\Mdef{\tC}{\tilde{C}}
\Mdef{\tE}{\tilde{E}}
\Mdef{\tH}{\tilde{H}}
\Mdef{\tK}{\tilde{K}}
\Mdef{\tL}{\tilde{L}}
\Mdef{\tM}{\tilde{M}}
\Mdef{\tN}{\tilde{N}}
\Mdef{\tP}{\tilde{P}}

\Mdef{\ft}{\tilde{f}}
\Mdef{\xt}{\tilde{x}}
\Mdef{\yt}{\tilde{y}}

\Mdef{\Ab}{\overline{A}}
\Mdef{\Bb}{\overline{B}}
\Mdef{\Cb}{\overline{C}}
\Mdef{\Db}{\overline{D}}
\Mdef{\Eb}{\overline{E}}
\Mdef{\Fb}{\overline{F}}
\Mdef{\Gb}{\overline{G}}
\Mdef{\Hb}{\overline{H}}
\Mdef{\Ib}{\overline{I}}
\Mdef{\Jb}{\overline{J}}
\Mdef{\Kb}{\overline{K}}
\Mdef{\Lb}{\overline{L}}
\Mdef{\Mb}{\overline{M}}
\Mdef{\Nb}{\overline{N}}
\Mdef{\Ob}{\overline{O}}
\Mdef{\Pb}{\overline{P}}
\Mdef{\Qb}{\overline{Q}}
\Mdef{\Rb}{\overline{R}}
\Mdef{\Sb}{\overline{S}}
\Mdef{\Tb}{\overline{T}}
\Mdef{\Ub}{\overline{U}}
\Mdef{\Vb}{\overline{V}}
\Mdef{\Wb}{\overline{W}}
\Mdef{\Xb}{\overline{X}}
\Mdef{\Yb}{\overline{Y}}
\Mdef{\Zb}{\overline{Z}}

\Mdef{\db}{\overline{d}}
\Mdef{\hb}{\overline{h}}
\Mdef{\qb}{\overline{q}}
\Mdef{\rb}{\overline{r}}
\Mdef{\tb}{\overline{t}}
\Mdef{\ub}{\overline{u}}
\Mdef{\vb}{\overline{v}}

\Mdef{\hc}{\hat{c}}
\Mdef{\he}{\hat{e}}
\Mdef{\hf}{\hat{f}}
\Mdef{\hA}{\hat{A}}
\Mdef{\hH}{\hat{H}}
\Mdef{\hJ}{\hat{J}}
\Mdef{\hM}{\hat{M}}
\Mdef{\hP}{\hat{P}}
\Mdef{\hQ}{\hat{Q}}

\Mdef{\thetab}{\overline{\theta}}
\Mdef{\phib}{\overline{\phi}}

\Mdef{\uA}{\underline{A}}
\Mdef{\uB}{\underline{B}}
\Mdef{\uC}{\underline{C}}
\Mdef{\uD}{\underline{D}}

\Mdef{\bolda}{\mathbf{a}}
\Mdef{\boldb}{\mathbf{b}}
\Mdef{\bfD}{\mathbf{D}}

\Mdef{\fm}{\frak{m}}
\Mdef{\fp}{\frak{p}}

\Mdef{\eps}{\epsilon}

\newcommand{\HBG}{H^*(BG)}
\newcommand{\HBK}{H^*(BK)}
\newcommand{\PP}{\mathbb{P}}

\newcommand{\piG}{\pi^G}
\newcommand{\CBG}{C^*(BG)}
\newcommand{\CBH}{C^*(BH)}
\newcommand{\CBK}{C^*(BK)}

\newcommand{\CBU}{C^*(BU)}

\begin{document}
\title{Borel cohomology and the relative Gorenstein condition for
 classifying spaces of compact Lie groups}

\author{J.P.C.Greenlees}
\address{Mathematical Institute, Zeeman Building
Coventry CV4 7AL. UK.}
\email{john.greenlees@warwick.ac.uk}

\date{}

\begin{abstract}
For a compact Lie group $G$ we show that if the  representing spectrum for 
Borel cohomology generates its category of modules if $G$ is
connected.  For a closed subgroup $H$ of $G$ we consider the map 
$\CBG\lra \CBH$ and establish the sense in which it is relatively
Gorenstein. Throughout, we pay careful attention to the importance of
connectedness of the groups. 
\end{abstract}

\thanks{I am grateful for to the Isaac Newton Institute for providing
  the environment to do this work and to EPSRC Grant Number EP/P031080/1 for
  support.  }
\maketitle

\tableofcontents
\section{Introduction}

\subsection{Context}
In this paper we consider a compact Lie group $G$ and continue the
study of $\CBG=C^*(BG;k)$ for a commutative ring $k$ as in
\cite{DGI}. We wish to use methods of commutative algebra, so we view
$\CBG$ as a commutative ring spectrum as in \cite{DGI}. More
precisely, $\CBG$ is the function spectrum $F(BG_+, Hk)$ from the
based space $BG_+$ to the Eilenberg-MacLane spectrum $Hk$ representing
ordinary cohomology with coefficients in $k$; the name is justified by
the fact that it is a commutative model for the cochains on $BG$ in
the sense that $\pi_*(\CBG)=\HBG$. 

To describe the result it is convenient to say that $G$ is 
{\em $k$-connected} if it is connected or if there is a prime $p$ so that 
$p^nk=0$ and $\pi_0(G)$ is a $p$-group. We say that a $d$-dimensional
real  representation 
$V$ of $G$ is {\em $k$-orientable} if the action of $G$ on $H_d(S^V;k)$
is trivial. We note that if $k$ is a field, every representation is $k$-orientable 
if $G$ is $k$-connected or in general if $k$  of characteristic 2. 

\subsection{The absolute case}
It is established in \cite{DGI} that the map $\CBG\lra k$ is
Gorenstein in the sense that if the adjoint representation $LG$ is
$k$-orientable  then there is an equivalence
$\Hom_{\CBG}(k, \CBG)\simeq \Sigma^{LG}k$ of $\CBG$-modules.
The suspension by $LG$ is simply suspension by $g=\dim (G)$, but 
writing it in this way gives a statement that is natural for automorphisms of $G$.

The statement is definitely false in general. The paper \cite{Lieca}
gives the example $G=O(2n)$, and one  easily checks that 
$$\Hom_{C^*(BO(2))}(\Q, C^*(BO(2))\simeq \Sigma^3\Q. $$

\subsection{The relative case}
The present paper begins the investigation of the relative case. We
suppose given a closed subgroup $H$, and  consider the associated map 
$$C^*(BG)\lra C^*(BH)$$
of commutative ring spectra. If $H$ is the trivial group this reduces
to the absolute case considered above.  We will show that the map 
is relatively Gorenstein, but in a sense that involves certain
twisting. 

If the representation  $LG$ of $G$ is $k$-orientable then Theorem
\ref{thm:relgor} states that there is an 
equivalence 
\begin{equation}
\label{eqn:relgor}
\Hom_{\CBG}(\CBH, \CBG )\simeq C^*(BH^{-L(G,H)}),   
\end{equation}
of $\CBG$-modules, where $L(G,H)=LG/LH$ is the representation of 
$W_G(H)=N_G(H)/H$ given by the conjugation action on the tangent space
to the identity coset of $G/H$. If in addition $LH$ is $k$-orientable
then this shows that $\CBG\lra \CBH$ is relatively Gorenstein of shift
$\dim (G/H)$ as one might hope. 

If we remove the condition that $LG$ is $k$-orientable then there is
an equivalence
\begin{equation}
\label{eqn:relgortw}
\Hom_{\CBG}(C^*(BH^{LG}), C^*(BG^{LG}))\simeq C^*(BH^{-L(G,H)}).    
\end{equation}
Note that this applies also in the case when $H=1$ where it gives an
appropriate statement when the adjoint representation is not
$k$-orientable. The reader is encouraged to see how this works 
for the case $G=O(2)$ and $k=\Q$ mentioned above. 

\subsection{Relationship to previous results}
Equivalence (\ref{eqn:relgor}) was asserted as Theorem 6.8 of
\cite{BGstrat} in general, but even when $G$ is $k$-connected, the proof 
given there is incomplete. Indeed, there are two significant
omissions. First, the argument given
there relied on the assertion that the Borel spectrum generates its
category of modules. It is explained in Sections \ref{sec:bgen} and
\ref{sec:Kuenneth} that this
follows from the convergence of an Eilenberg-Moore spectral
sequence when $G$ is $k$-connected and that it is not true in general.
This completes the proof of Theorem \ref{thm:relgor} in the case that $G$ is
$k$-connected. We then observe in Section \ref{sec:relgor} that
Equivalence (\ref{eqn:relgortw}) follows by embedding $G$ in a connected group $U$ and then using the 
results for $G\subseteq U$ and $H\subseteq U$. Finally, Theorem 6.8 of
\cite{BGstrat} is not true in general when $LG$ is not orientable,
even when $H=1$, as pointed out above.

The author observed the gap in the earlier proof when considering the
general Gorenstein properties of the map $C^*(BG)\lra C^*(BH)$.
The formulation of the Gorenstein property in Equation (\ref{eqn:relgor})
is rather straightforward.  However for many purposes, one is more
interested in the Gorenstein {\em duality} property which leads to a
suitable  local cohomology spectral sequence. There is a well understood
equivariant approach to this (the local cohomology theorem for
the family of subconjugates of $H$),  but the commutative algebra, Morita
theory and orientability involves significant new complications
 that the author intends to return to elsewhere

\subsection{Borel cohomology}
We are concerned with Borel cohomology with coefficients in a
commutative ring $k$, defined on based
$G$-spaces $X$ by 
$$b_G^*(X)=H^*(EG\times_GX, EG\times x_0;k)=H^*(EG_+\sm_GX;k). $$
From the isomorphisms
$$b_G^*(X)=H^*(EG_+\sm_GX)\cong [EG_+\sm_GX,Hk]^*
\cong [EG_+\sm X, Hk]_G^*\cong [X, F(EG_+, Hk)]_G^*$$
it is apparent that the representing spectrum is given by
$b=F(EG_+,Hk)$. As first pointed out in \cite{EM97} this can be realized as a commutative orthogonal
spectrum, since $Hk$ is a commutative ring and $EG$ is a space with a
 diagonal. 

The main substance of this paper is the proof of  some basic facts that are useful
in studying the homotopy category of $b$-module $G$-spectra. The results
generally have counterparts in more classical approaches, many of
which are familiar. 

\subsection{Conventions}
We say that a space  $B$ is {\em $k$-simply connected} if $B$ is simply 
connected or if $p^nk=0$ for some prime $p$ and  $n\geq 1$, 
the space $B$ is connected   and $\pi_1(B)$ is a finite $p$-group. 

We also say that a group $G$ is {\em $k$-connected} if it is connected or if 
$p^nk=0$ for some prime $p$ and  $n\geq 1$ and $\pi_0(G)$ is a finite
$p$-group. If $G$ is $k$-connected then $BG$ is $k$-simply connected. 

\section{The K\"unneth-Eilenberg-Moore spectral sequence}
Ordinary derived homological algebra lifts to non-equivariant spectra 
very well \cite{EKMM} to give a K\"unneth spectral sequence 
$$\Tor^{R_*}_{*,*}(M_*, N_*)\Rightarrow (M\tensor_RN)_*$$
for any commutative ring spectrum $R$ and $R$-modules $M$ and $N$. This works 
because we can realize a $R_*$ resolution by the ring $R$ (we will run
through the argument in Section \ref{sec:Kuenneth} below); because $R$ is 
a {\em generator} of the category of $R$-modules, 
 the spectral sequence always converges. 

Equivariantly, we can make the same construction, but usually the 
spectral sequence will not converge. This is because if $R$ is a 
$G$-spectrum we need all the modules $G/H_+\sm R$ to generate the 
category of $R$-modules. The point of the following theorem is that 
 the ring $G$-spectrum $b$ {\em does}
generate the category of all $b$-modules provided $G$ is 
$k$-connected. This is a very special feature of Borel cohomology.

\begin{thm}
\label{thm:EMmain}
If $b$ represents Borel cohomology with coefficients in $k$  and $G$ is 
$k$-connected then for any $b$-modules $M$, $N$ there is a strongly
convergent spectral sequence 
$$\Tor^{\HBG}_{*,*}(\piG_*M, \piG_*N)\Rightarrow
\piG_*(M\tensor_bN),  $$
and a conditionally convergent spectral sequence
$$\Ext_{\HBG}^{*,*}(\piG_*M, \piG_*N)\Rightarrow \piG_*(\Hom_b(M,N)). $$
\end{thm}

A case of particular interest is  $N=F(G/K_+,b)$, where we have 
$$M\tensor_bN= M\tensor_b F(G/K_+, b)\simeq F(G/K_+,M)$$
since $G/K_+$ is small as as a $G$-spectrum. 

\begin{cor}
\label{cor:main}
If the group $G$ is $k$-connected then for any subgroup $K\subseteq
G$,  there is a strongly convergent spectral sequence 
$$\Tor^{\HBG}_{*,*}(b^*_G(X), \HBK)\Rightarrow b_K^*(X). $$
\end{cor}

\begin{remark}
\label{rem:kconnEM}
(i) The hypothesis that $G$ is $k$-connected is necessary to get a general 
statement.  For example, if $K=G_e$
is the identity component of $G$, and $k$ is a field of characteristic
0 then $b^*_G(X)=b^*_K(X)^{G_d}$ where $G_d=G/G_e$ is the discrete 
quotient. Working rationally, we may realise a non-zero simple
$G_d$-module $V$ as a Moore spectrum $MV$ and hence  get a zero
$E_2$-term, whilst $b_{G_e}^*(MV)\neq 0$.

(ii) The corollary takes the form of an Eilenberg-Moore spectral
sequence in the sense that it has the expected $E_2$-term and end
point. In fact the proof will show that the entire spectral sequence
coincides with the classical Eilenberg-Moore spectral sequence for the
pullback square 
$$\xymatrix{
EK\times_K Z \ar[r]\ar[d]&EG\times_G Z\ar[d]\\
BK \ar[r]&BG
}$$
where $Z$ is a $G$ space and  $X=Z_+$. 
\end{remark}

\section{Generating the category of $b$-modules}
\label{sec:bgen}
Just as every $G$-spectrum is built out of cells $G/K_+$ as $K$ varies
through closed subgroups, so any $b$-module is built out of the
extended $b$-modules $G/K_+\sm b$. 

We note that $[X, F(G/K_+, b)]_G^*=b_K^*(X)$ so that if $F(G/K_+,b)$
is built from $b$ then $b_G^*(X)=0$ implies $b_K^*(X)=0$ if $X$ is
finite. Such an implication does not hold in general, but from 
 the Eilenberg-Moore theorem 
recalled in Remark \ref{rem:kconnEM} (ii), we see 
$$C^*(EK\times_K Z)\simeq C^*(EG\times_G Z)\tensor_{\CBG}\CBK$$
provided $G$ is connected. This may make the following statement
plausible. 

\begin{thm}
\label{thm:genmain}
If $G$ is $k$-connected and $b$ represents Borel cohomology with
coefficients in $k$ then $b$ is a generator of the category of
$b$-modules. 
\end{thm}

\begin{remark}
Note that there is no finiteness requirement. The question of which
$b$-modules are {\em finitely} built by  $b$ is more subtle. In
\cite{Dsg} the notion of finite generation is introduced, which is
obviously a necessary condition for a module to be finitely built by
$b$. It is shown that for finite groups  $b$ 
finitely builds every finitely generated $b$-module precisely when
$G$ is $p$-nilpotent. 
\end{remark}

The proof of the theorem  is by comparison with classical
Eilenberg-Moore spectral sequence, and the agreement of the
spectral sequences may be of interest in itself. In Section
\ref{sec:torus} we outline an
alternative approach via  reduction to unitary groups and then the maximal torus.

\subsection{Consequence of $b$ generating}
As noted above there are K\"unneth and Universal Coefficient spectral
sequences for the (non-equivariant) fixed point ring spectrum $b^G$:  
$$\Ext_{H^*(BG)}^{*,*}(M^G_*, N^G_*) =\Ext_{b_G^*}^{*,*}(M^G_*, N^G_*)\Rightarrow 
\pi_*(\Hom_{b^G}(M^G, N^G)). $$
$$\Tor^{H^*(BG)}_{*,*}(M^G_*, N^G_*) =\Tor^{b_G^*}_{*,*}(M^G_*, N^G_*)\Rightarrow 
\pi_*(M^G\tensor_{b^G} N^G). $$

The point is that when $b$ is a generator of the category of
$b$-module $G$-spectra, these also calculate the homotopy of the
$G$-spectra $\Hom_b(M,N)$ and
$M\tensor_bN$ respectively, since in that case the following 
two elementary equivalences apply to all modules. 

\begin{lemma}
\label{lem:HomTensor}
(i) The map 
$$\Hom_b(M, N)^G\lra \Hom_{b^G}(M^G,N^G)$$
is an equivalence if $M$ is built by $b$. Accordingly, there is then a
spectral sequence
$$\Ext_{H^*(BG)}^{*,*}(M^G_*, N^G_*)\Rightarrow \pi^G_*(\Hom_b(M,N)).$$

(ii) The natural map 
$$ M^G\tensor_{b^G}N^G \lra (M\tensor_b N)^G $$
is an equivalence if $M$ is built by $b$. Accordingly, there is then a
spectral sequence
$$\Tor^{H^*(BG)}_{*,*}(M^G_*, N^G_*)\Rightarrow \pi^G_*(M\tensor_bN).\qqed$$
\end{lemma}

\begin{proof}
In both cases, the map is obviously an equivalence when $M=b$. The
class of objects for which it is an equivalence is closed under
coproducts, integer suspensions and mapping cones. 
\end{proof}

\section{The K\"unneth spectral sequence}
\label{sec:Kuenneth}
In this section we prove Theorems \ref{thm:EMmain} and \ref{thm:genmain}. The main point of the
proof of the Eilenberg-Moore-type spectral sequence is that the usual
construction gives a {\em convergent} spectral
sequence. Applying the spectral sequence in the form of Corollary
\ref{cor:main} shows that if $M^G_*=0$ then $M^K_*=0$ for all
subgroups $K\subseteq G$ and hence that $b$ generates all
$b$-modules. 

\subsection{The construction}
First,  we construct a resolution of $\pi^G_*(M)$ by free
$\HBG=\pi^G_*(b)$-modules
$$0\lla M^G_*\lla P_0\lla P_1\lla P_2\lla \cdots . $$
We  then realize it by a diagram 
$$\xymatrix{
M=M_0\ar[r]  &M_1 
\ar[r]                    &M_2 \ar[r]                  &M_3 \ar[r] & \\
\PP_0\ar[u]    &\Sigma^1 \PP_1\ar[u] &\Sigma^2\PP_2\ar[u] &\Sigma^3 \PP_3\ar[u] &
}$$
where $\pi^G_*(\PP_s)\cong P_s$.  Defining $M^s$ by  the cofibre sequence 
$$M^s\lra M \lra M_s$$
we see that $M^s$ and $M^{\infty}=\hocolim_s M^s$ are built from $b$, and that there is a map
$\alpha_M: M^{\infty} \lra M$ which is a
$\pi^G_*$-isomorphism. Furthermore the map is unique over $M$. 
This induces a map 
$$\alpha_M\tensor_b N:M^{\infty}\tensor_b N \lra M\tensor_b N, $$
which we will prove is also a $\piG_*$-isomorphism. The point is that
the spectral sequence obtained by filtering $M^{\infty}$  then reads
$$\Tor^{\HBG}_{*,*}(\piG_*(M^{\infty}), \piG_*(N))\Rightarrow \pi^G_*(M^{\infty}\tensor_b 
N)=\pi^G_*(M\tensor_b N)$$
as required. 


It remains to show that $\alpha_M\tensor_b N: M^{\infty}\tensor_bN \lra
M\tensor_bN$ is a $\piG_*$-isomorphism. Define 
$$\cC_N:=\{ M \st \alpha_M \tensor_bN \mbox{ is a $\piG_*$ iso }\}$$
$$\mcD_M:=\{ N \st \alpha_M\tensor_bN\mbox{ is a $\piG_*$ iso }\}$$
We note that $\cC_N$ is a  localizing subcategory for any $N$ and 
 $\mcD_M$ is a  localizing subcategory for any $M$.

We will show in Subsection \ref{subsec:KEM} below that $\alpha_M\tensor N$ is a $\piG_*$-isomorphism for
$M=F(X,b)$ and $N=F(G/K_+,b)$. Since the spectra $F(G/K_+,b)$ as $K$
varies generate all $b$-modules, this shows that the category
$\mcD_{F(X,b)}$  is the whole category of $b$-modules.  It follows
that for each $N$, the category $\cC_N$ contains $F(X,b)$. Since the
modules $F(X,b)$ generate all $b$-modules, we infer $\alpha_M\tensor_bN$ is a
$\piG_*$ isomorphism for all $M, N$.

This completes the formal part of the proof, and leaves us with the
substance.
  It remains to show that $M^{\infty}\tensor_bN\lra M\tensor_bN$
is a $\piG_*$-isomorphism when $M=F(G/K_+, b), N=F(X,b)$. Indeed, 
this is the map $F(G/K_+, M^{\infty})\lra F(G/K_+,M)$ which in $G$
fixed points reads $\pi_*^K(M^{\infty})\lra \pi_*^KF(X,b)=H^*_K(X)$. We will observe that
we can choose the resolution so that the filtration on $(M^{\infty})^K$ is the
Eilenberg-Moore filtration, and then the isomorphism follows from 
the  convergence of the Eilenberg-Moore spectral sequence, which is the main 
result of \cite{DwyerEM}. 

\subsection{Comparison with the classical Eilenberg-Moore spectral sequence}
\label{subsec:KEM}
We observe that suitable constructions of the Eilenberg-Moore and
K\"unneth spectral sequences coincide. 

To start with, if we are given maps  $X\stackrel{f}\lra B \stackrel{g}
\lla Y$ of spaces, we may form the two-sided geometric cobar complex 
$$CB^n(X,B,Y)=X\times B^{\times n}\times Y$$
with coface maps
$$\begin{array}{rcl}d^0(x, b_1, \ldots , b_n, y)&=&(x, f(x), b_1, \ldots , b_n, y)\\
d^i(x, b_1, \ldots , b_n, e)&=&(x, b_1, \ldots, *,\ldots  , b_n, e)
                                \mbox{ for $1\leq i\leq n$ }\\ 
d^{n+1}(x, b_1, \ldots , b_n, e)&=&(x,  b_1, \ldots , b_n, g(y),
y). 
\end{array}$$
The $i$th codegeneracy map is given by projection away from the
$i+1$st factor of $B$.

The Eilenberg-Moore spectral sequence of the pullback square 
$$\xymatrix{
P \ar[r]\ar[d]&X\ar[d]^f\\
Y\ar[r]_g&B 
}$$
where $f$ or $g$ is a fibration, is obtained by taking cochains of 
$CB(X,B,Y)$, which is equivalent to  the bar construction
$B(C^*(X),C^*(B), C^*(Y))$ if $X,B$ and $Y$ are locally finite. 

More precisely, if $A$ is a strictly commutative ring spectrum, we
write $C^*(X;A)=F(X_+,A)$,  and note that we may form the simplicial
spectrum $B(C^*(X;A),C^*(B;A),C^*(Y; A))$ and there is a map 
$$B(C^*(X;A), C^*(B;A), C^*(Y;A)) \lra C^*(CB(X,B,Y), A) $$
which is a weak equivalence if $B,X, Y$ are finite complexes, or  if
they are locally finite and $A$ represents ordinary cohomology.

Now consider the special case 
$$\xymatrix{
EG\times_KZ \ar[r]\ar[d]&EG\times_G Z\ar[d]\\
BK\ar[r]&BG 
}$$

This is an instance of the above, so we can form the cosimplicial space
$CB(EG\times_GZ,BG,BK)$. We can also form the cosimplicial $G$-space
$CB(EG\times Z,BG,BK)$. It is then immediate that we have an
isomorphism 
$$F(CB(EG\times_GZ,BG,BK)_+, Hk)\simeq
F(CB(EG\times Z,BG,BK)_+, b)^G $$
of simplicial spaces. 

Finally, we observe that the cosimplicial filtration of $CB(EG\times_G
Z,BG,EG)$ gives a filtration of $F(CB(EG\times_G Z,BG,EG)_+, b) $ which
is a particular example of a filtration giving rise to the K\"unneth
spectral sequence. 

Now the resolution is $B(C^*(X_{hG}), \CBG, b)$, so the spectral
sequence is obtained by applying tensoring with $N=F(G/K_+,b)$  
$$ B(C^*(X_{hG}), \CBG, b)\tensor_{b}F(G/K_+,b)\simeq 
B(C^*(X_{hG})), \CBG, F(G/K_+,b)). $$
The spectral sequence is obtained by passing to fixed points 
to obtain $B(C^*(X_{hG}), \CBG, \CBK)$ and then taking homotopy. This is
precisely the Eilenberg-Moore spectral sequence as required.

\section{Tales of the torus}
\label{sec:torus}
In this section we describe a method based on the properties of the
torus. It is standard practice in transformation groups to have
a statement $P(G)$ for each compact Lie group $G$. We then attempt a 
proof of $P(G)$ in general by the strategy
\begin{itemize}
\item $P(G)$ is true with  $G=T(n)$ an $n$-torus
\item $P(G)$ for $G=U(n)$ follows from $P(T(n))$ where $T(n)$ is the
  maximal torus of $U(n)$. 
\item If $G$ is a subgroup of $U(n)$ then $P(G)$ follows from $P(U(n))$.
\end{itemize}

In our case we are interested in a statement that is not true for all
groups, so the reduction cannot work in generality. However the first
and second steps work, and we can investigate the third. 

\begin{thm}
\label{thm:Tngen}
If $G$ is a torus, then for any coefficient ring $k$, the category of 
$b$-modules is generated by $b$. 
\end{thm}

\begin{proof} Suppose  $G$ is a torus. 
We simply need to show if $M$ is a $b$-module with $\pi^G_*M=0$ then 
in fact 
$\pi^H_*M=0$ for all $H$ (and hence $M\simeq 0$).

For any subgroup $H$ of codimension $c$ we may find one dimensional
representations $\alpha_1, \ldots, \alpha_c$ so that
$H=\ker(\alpha_1)\cap \ldots \cap \ker(\alpha_c)$, and then 
$$G/H=S(\alpha_1)\times \cdots \times S(\alpha_c). $$
Now observe that $S(\alpha)=G/\ker (\alpha)$, and we have a  cofibre sequence
$$S(\alpha)_+\lra S^0\lra S^{\alpha}. $$
Smashing with a $b$-module $M$ this gives
$$S(\alpha)_+\sm M \lra S^0\sm M \lra S^{\alpha}\sm M\simeq S^2\sm M, 
$$
where the last equivalence is a Thom isomorphism. 
This shows that if $\pi^G_*(M)=0$ then $\pi^G_*(G/\ker (\alpha)_+ \sm
M)=0$

Repeating, we see $\pi^H_*(M)=0$. 
\end{proof}

\begin{thm}
\label{thm:Ungen}
If $G=U(n)$, then for any coefficient ring $k$, the category of 
$b$-modules is generated by $b$. 
\end{thm}

The following observation will be useful. 

\begin{lemma}
If  $G\supseteq H\supseteq K$ and $b_H^*$ is a retract of $b_K^*$ then
$M^H_*$ is a retract of $M^K_*$ for all $b$-modules $M$.  
\end{lemma}

\begin{proof}
The hypothesis is that the $H$-spectrum $b$ is a retract of $F(H/K_+, b)$,
and hence $F(G/H_+, b)$ is a retract of $F(G/K_+,b)$. Tensoring with
$M$ and using the equivalence $F(G/H_+, b)\tensor_bM\simeq F(G/H_+,M)$
we obtain the desired statement. 
\end{proof}

It is well known that $b_{T(n)}^*=k[x_1, \ldots, x_n]$ is free over  
$$b_{U(n)}^*=k[c_1, \ldots , c_n]=k[x_1, \ldots, x_n]^{\Sigma_n}. $$
The K\"unneth Theorem (a special case of Theorem \ref{thm:EMmain})
therefore takes an elementary form, which admits a simple  proof. 

\begin{lemma}
\label{lem:UTEM}
There is a natural isomorphism
$$b_{T(n)}^*(X)=b_{U(n)}^*(X)\tensor_{b_{U(n)}^*}b_{T(n)}^*$$
\end{lemma}

\begin{proof}
There is a natural transformation of equivariant cohomology theories
from the right to the left. 
The classical Eilenberg-Moore Theorem states that it is an isomorphism
for spaces $X$. This includes the cells $G/K_+$, so the natural
transformation is an isomorphism in general. 
\end{proof}

\begin{remark}
It would be nice to have a proof based on a general calculation that
the $U(n)$-space  $U(n)/K \times U(n)/T(n)$ has a nice cell structure for all $K$.
\end{remark}

We are finally ready to prove Theorem \ref{thm:Ungen}.

\begin{proof}
If $M$ is a $U(n)$-equivariant $b$-module with $M^{U(n)}_*=0$ then by
Lemma \ref{lem:UTEM} we have $M^{T(n)}_*=0$. By Theorem \ref{thm:Tngen}
it follows that $M^K_*=0$ for all $K\subseteq T(n)$. 

Now if $e$ is the idempotent of the Burnside ring supported on
subgroups of the maximal torus we have $eM\simeq 0$ and $eb\simeq b$
so  
$$M\simeq b\tensor_b M\simeq eb\tensor_b M\simeq b\tensor_b
eM\simeq 0$$
as required. 
\end{proof}

\begin{remark}
We now consider the deduction of the statement for a general group
$G$, which we know is true if $G$ is $k$-connected. First, we may 
embed $G$ in $U(n)$ for some $n$. 

Now suppose $\pi^G_*(M)=0$ and note that 
$\pi^{U(n)}_*(F_G(U(n)_+, M))=\pi^G_*(M)=0$. From the case of $U(n)$ we infer that 
$F_G(U(n)_+, M )\simeq *$ (first $U(n)$-equivariantly and then
$G$-equivariantly by restriction). 

There is always a $G$-map $F_G(U(n)_+, M)\lra F_G(G_+, M)$, but it
need not be split. We can attempt to find a splitting by considering a 
$G$-cell structure of $U(n)$. However if $M=b$ we see that a first
obstruction is the surjectivity of the restriction map $H^*(BU(n))\lra
H^*(BG)$, and of course it often happens that $H^*(BG)$ is not 
generated by Chern classes. 
\end{remark}

\section{The relative Gorenstein condition}
\label{sec:relgor}

We are ready to give the proof of the relative Gorenstein property. To
obtain a natural statement we note that $L(G,H)=LG/LH$ is a representation
of $N_G(H)$, typically non-trivial on $H$. 

\begin{thm}
\label{thm:relgor}
For any $k$-connected compact Lie group $G$ and any closed subgroup
$H$ we have an equivalence of $\CBG$-modules with an action of $\pi_0(W_G(H))$:
$$\Hom_{C^*(BG)}(C^*(BH), C^*(BG))\simeq C^*(BH^{-L(G,H)}).   $$
The equivalence holds for arbitrary $G$ if $LG$ is $k$-orientable and $LH$ is
$k$-orientable. Without the $k$-connectedness or orientability
hypotheses,  we have the equivalence
$$\Hom_{\CBG}(C^*(BH^{LG}), C^*(BG^{LG})) \simeq C^*(BH^{-L(G,H)}).$$
\end{thm}

\begin{remark}
The first equivalence does not hold in general. This is familiar from the 
case $H=1$ where \cite{Lieca} gives the example $G=O(2n)$. The first
case is very easy to check: 
$$\Hom_{C^*(BO(2))}(k, C^*(BO(2))\simeq \Sigma^3k. $$
\end{remark}

The failure of the first equivalence in the theorem is tied to the fact that not all
bundles are trivial, not all Thom spaces have cohomology free over
the base, and the fixed points of a suspension need not be a shift of
the original.

The positive result is as follows. 
 \begin{lemma}
\label{lem:susp}
If $G$ is $k$-connected then for any representation $W$ of $G$ we have
an equivalence
$$\Hom_{b^G}(M^G, N^G)\simeq \Hom_{b^G}((\Sigma^W M)^G, 
(\Sigma^WN)^G)$$
\end{lemma}

\begin{remark}
To see this does not hold without some connectedness hypothesis, consider finite groups $G$ and
work over the rationals. 
\end{remark}
\begin{proof}
By Theorem \ref{thm:genmain} and 
Lemma  \ref{lem:HomTensor} (i)  we have
$$\Hom_{b^G}(M^G,N^G)\simeq \Hom_b(M,N)^G .$$
Now use the fact that suspension is an equivalence of the stable
homotopy category, and take fixed points again.  
\end{proof}

We may now turn to the proof of the theorem. 

\begin{proof}
Since $G/H_+$ is a finite spectrum, we have an equivalence
$$\Hom_b(F(G/H_+, b),b)\simeq G/H_+\sm b. $$
Taking fixed points we have
$$\Hom b(F(G/H_+, b),b)^G
\simeq (G/H_+\sm b)^G\simeq (\Sigma^{L(G,H)}
b)^H\simeq C^*(BH^{-L(G,H)}). $$

We may now complete the proof if $G$ is connected. In that case Theorem
\ref{thm:genmain} gives equivalences
\begin{multline*}
\Hom_b(F(G/H_+, b),b)^G \stackrel{\simeq }\lra  \Hom_{b^G}(F(G/H_+,
b)^G,b^G)\simeq \\  \Hom_{b^G}(b^H,b^G)\simeq \Hom_{\CBG}(\CBH,
\CBG),\end{multline*}
as in Lemma \ref{lem:HomTensor} (i). 

Finally, we note that this permits us to deduce the general case.

Choose a faithful representation of $G$ in $U=U(n)$ for some $n$. 
From the connected case we have
$$\Hom_{\CBU}(\CBG, \CBU)\simeq C^*(BG^{-L(U,G)})$$
and 
$$\Hom_{\CBU}(\CBH, \CBU)\simeq C^*(BH^{-L(U,H)}). $$

Now we have
$$\begin{array}{rcl}
\Hom_{\CBG}(\CBH, C^*(BG^{-L(U,G)})) &\simeq &\Hom_{\CBG}(\CBH, 
                                  \Hom_{\CBU}(\CBG, 
                                  \CBU))\\
&\simeq &\Hom_{\CBU}(\CBH, \CBU)\\
&\simeq &C^*(BH^{-L(U,H)}). 
\end{array}$$
as required. 

This completes the proof if $LG$ and $LH$ are $k$-orientable. 
Because $L(U,G)$ need not be trivial for $G$,  some additional
work is necessary to go further. To start with, 
note that $L(U,H)$ does admit an action of $N_U(H)$ and hence
$N_G(H)$.


We start with the equivalence of $U$-spectra
$$U_+\sm_H(\Sigma^{-L(U,G)}b)\simeq \Hom_b(F_H(U_+, \Sigma^{L(U,G)}b),
b), $$
and take $U$-fixed points.   Since $U$ is connected,  by
\ref{thm:genmain} and \ref{lem:HomTensor} this gives the first
equivalence of the following:
$$\begin{array}{rcl}
(\Sigma^{L(U,H)-L(U,G)}b)^H 
&\simeq &\Hom_{b^U}((\Sigma^{L(U,G)}b)^H,b^U) \\
      &\simeq &\Hom_{b^G}((\Sigma^{L(U,G)}b)^H,\Hom_{b^U}(b^G,b^U))\\
 &\simeq   &\Hom_{b^G}((\Sigma^{L(U,G)}b)^H,(\Sigma^{L(U,G)}b)^G) 
\end{array}$$
Since $L(U,G)=LU/LG$, $L(U,H)=LU/LH$, Lemma \ref{lem:susp} permits us
to deduce the second equivalence of the theorem.
\end{proof}


\end{document}